\documentclass{amsart}

\usepackage[cp1252]{inputenc}
\usepackage[T1]{fontenc}
\usepackage{lmodern}
\usepackage{graphicx}
\usepackage[all]{xy}
\usepackage{txfonts}
\usepackage{amsmath,amsthm}

\usepackage{pstricks,pstricks-add,pst-math,pst-xkey}

\newcommand{\K}{\mathbb{K}}
\newcommand{\p}{\mathbb{P}^2}
\newcommand{\Aut}[1]{\textbf{Aut}({#1})}

\newcommand{\A}{\mathbb{A}}

\newcommand{\compl}{\p \backslash}

\newcommand{\E}{\tilde{E}}

\newcommand{\Q}{\tilde{Q}}
\newcommand{\C}{\tilde{C}}

\numberwithin{equation}{section}%

\newtheorem{thm}[equation]{Theorem}%

\newtheorem{lemme}[equation]{Lemma}

\newtheorem{conj}[equation]{Conjecture}

\newtheorem{rem}[equation]{Remark}

\newtheorem{dfn}[equation]{Definition}

\begin{document}
\title{New distinct curves having the same complement in the projective plane.}
\author{Paolo Costa}
\address{Paolo Costa, UniGe}
\email{costapa5@etu.unige.ch}
\date{\today}
\begin{abstract}
In 1984, H. Yoshihara conjectured that if two plane irreducible curves have isomorphic complements, they are projectively equivalent, and proved the conjecture for a special family of unicuspidal curves. Recently, J. Blanc gave counterexamples of degree $39$ to this conjecture, but none of these is unicuspidal.
In this text, we give a new family of counterexamples to the conjecture, all of them being unicuspidal, of degree $4m+1$ for any $m\ge 2$. In particular, we have counterexamples of degree $9$, which seems to be the lowest possible degree.
\end{abstract}
\maketitle{}
\section{The conjecture.}

In the sequel, we will work with algebraic varieties over a fixed ground field $\K$, which can be arbitrary.
\\
\begin{conj}[\cite{Yos84}]
Suppose that the ground field is algebraically closed of characteristic zero. Let $C \subset \p$ be an irreducible curve. Suppose that $\p \backslash C$ is isomorphic to $\p \backslash D$ for some curve $D$. Then $C$ and $D$ are projectively equivalent, i.e. there's an automorphism $\alpha : \p \to \p$ such that $\alpha(C)=D$.
\end{conj}

This conjecture leads to several alternatives. Let $\psi : \p \backslash C \to \p \backslash D$ be an isomorphism.\
If the conjecture holds, then :
\begin{itemize}
	\item either $\psi$ extends to an automorphism of $\p$ and we can choose $\alpha:=\psi$.
	\item or $\psi$ extends to a strict birational map $\psi : \p \dasharrow \p$. In this case, there's an automorphism $\alpha : \p \to \p$ such that $\alpha(C)=D$.
\end{itemize}
Otherwise, if $\psi$ gives a counterexample to the conjecture, then :
\begin{itemize}
	\item either $C$ and $D$ are not isomorphic.
	\item or $C$ and $D$ are isomorphic, but not by an automorphism of $\p$.
\end{itemize}

In this text, we are going to study the conjecture in the case of curves of type I.\\

\begin{dfn} We say that a curve $C \subset \p$ is of \textbf{type I} if there's a point $p \in C$ such that $C \backslash p$ is isomorphic to $\A^1$.\\
We say that a curve $C \subset \p$ is of \textbf{type II} if there's a line $L \subset \p$ such that $C \backslash L$ is isomorphic to $\A^1$.\\
\end{dfn}
All curves of type II are of type I, but the converse is false in general. Moreover, a curve of type I is a line, a conic, or a unicuspidal curve (a curve with one singularity of cuspidal type).

In the case of curves of type II, H. Yoshihara showed that the conjecture is true \cite{Yos84}, but in general the conjecture doesn't hold. Some counterexamples are given in \cite{Bla09}, but these curves are not of type I.

In this article, we give a new family of counterexamples, of degree $4m+1$ for any $m\ge 2$. These are all of type I, and some of them have degree $9$, which seems to be the lowest possible degree (see the end of the article for more details). In Section 2 we give a general way to constuct examples, that we precise in Section 3. The last section is the conclusion.\\

I would like to thank J. Blanc for asking me the question and for his help during the preparation of this article. I also thank T. Vust for interesting discussions on the result.

\section{The Construction.}

We begin with giving a general construction, which provides isomorphisms of the form $\compl C \to \compl D$ where $C,D$ are curves in $\p$. We start with the following definition :
\begin{dfn}
We say that a morphism $\pi : S \to \p$ is a \textbf{$(-1)-$tower resolution} of a curve $C$ if :
\begin{enumerate}
	\item $\pi=\pi_1 \circ ... \circ \pi_m$ where $\pi_i$ is the blow-up of a point $p_i$,
	\item $\pi_i(p_{i+1})=p_i$ for $i=1,...,m-1$,
	\item the strict transform of $C$ in $S$ is a smooth curve, isomorphic to $\mathbb{P}^1$, and has self-intersection $-1$.
\end{enumerate}
\end{dfn}

The isomorphisms of the form $\compl C \to \compl D$ are closely related to $(-1)-$tower resolutions of $C$ and $D$ because of the following Lemma :

\begin{lemme}[\cite{Bla09}] \label{Bla09}
Let $C \subset \p$ be an irreducible algebraic curve and $\psi : \compl C \to \compl D$ an isomorphism. Then, either $\psi$ extends to an automorphism of $\p$, either it extends to a strict birational transform $\phi : \p \dasharrow \p$.\\
Consider the second case. Let $\chi : X \to \p$ a minimal resolution of the indeterminacies of $\phi$, call $\E_1,...,\E_m$ and $\C$ the strict transforms of its exceptional curves and $C$ in $X$ and set $\epsilon:=\phi \circ \chi$. Then :
\begin{enumerate}
	\item $\chi$ is a $(-1)-$tower resolution of $C$
	\item $\epsilon$ collapses $\C, \E_1,..., \E_{m-1}$ and $\epsilon(\E_m)=D$,
	\item $\epsilon$ is a $(-1)-$tower resolution of $D$.
\end{enumerate}
\end{lemme}

\begin{rem}
This lemma shows that if $C$ doesn't admit a $(-1)-$tower resolution, then every isomorphism $\compl C \to \compl D$ extends to an automorphism of $\p$. So counterexamples will be given by rational curves with only one singularity.
\end{rem}

We start with a smooth conic $Q \subset \p$ and $\phi \in \Aut{\compl Q}$ which extends to a strict birational map $\phi : \p \dasharrow \p$. Call $p_1,...,p_m$ the indeterminacies points of $\phi$; according to Lemma~$\ref{Bla09}$, we can order the points so that $p_1$ is a point of $\p$ and $p_i$ is infinitely near to $p_{i-1}$ for $i=2,...,n$. Consider $\chi : X \to \p$, a minimal resolution of the indeterminacies of $\phi$ and set $\epsilon:=\phi \circ \chi$. Lemma~$\ref{Bla09}$ says that :
\begin{enumerate}
	\item $\chi$ is a $(-1)-$tower resolution of $Q$,
	\item $\epsilon$ collapses $\Q, \E_1,..., \E_{m-1}$ and $\epsilon(\E_m)=Q$,
	\item $\epsilon$ is a $(-1)-$tower resolution of $Q$.
\end{enumerate}
Now, consider a line $L \subset \p$, which is tangent to $Q$ at $p \neq p_1$. Since $\phi$ contracts $Q$, then $C:=\phi(L)$ is a curve with an unique singular point which is $\phi(Q)$. Since $L \cap (\compl Q) \simeq \A^1$, we have $C \cap (\compl Q) \simeq \A^1$, which means that $C$ is of type I.

Consider now a birational map $f \in \Aut{\compl L}$ which extends to a strict birational map $\p \dasharrow \p$ and satisfies :
\begin{enumerate}
	\item $f(Q)=Q$,
	\item $f(p_1)=p_1$.
\end{enumerate}

Now, we are going to get a new birational map $\phi' : \p \dasharrow \p$ which restricts to an automorphism of $\compl Q$ using the $p_i$'s and $f$. Set :
$$p_i':=f(p_i).$$
Note that $p_i'$ is a well-defined point infinitely near to $p_{i-1}'$ for $i>1$.\\
Let's call $\chi' : X' \to \p$ the blow-up of the $p_i'$'s and $\E_1',...,\E_m'$ and $\Q'$ the strict transforms of the exceptional curves of $\chi'$ and of $Q$ in $X'$.\\
Since $f(Q)=Q$ and $f$ is an isomorphism at the neighbourhood of $p_1$, the intersections between $\E_1,...,\E_m$ and $\Q'$ are the same as those between $\E_1,...,\E_m$ and $\Q$. Then there's a morphism $\epsilon' : X' \to \p$ which contracts $\E_1',...,\E_{m-1}'$ and $\Q'$. Moreover, $\epsilon'(\E_m')$ is a conic, and up to isomorphism we can suppose that $\epsilon'(\E_m')=Q$.
\\
By construction, the birational map $\phi'$ restricts to an automorphism of $\compl Q$. In fact, neither of the $p_i'$'s belongs to $L$ (as proper or infinitely near point), so $\phi'(L)$ is well defined. Moreover, $\phi'$ collapses $Q$, so $D:=\phi'(L)$ is a curve with an unique singular point which is $\phi'(Q)$.\\
Set then $\psi:=\phi' \circ f \circ \phi^{-1}$. We have the following commutative diagram :\\
\[\xymatrix@R=3.5mm{
	 & X' \ar[rd]^{\epsilon'} \ar[ld]_{\chi'} &  \\
	\p \ar@{-->}[rr]^{\phi'} & & \p \\
	 & X \ar[rd]^{\epsilon} \ar[ld]_{\chi} &  \\
	\p \ar@{-->}[rr]^{\phi} \ar@{-->}[uu]^{f} & & \p \ar@{-->}[uu]_{\psi} \\
	}
\]
\begin{lemme}
The map $\psi : \compl C \to \compl D$ induced by the birational map defined above is an isomorphism.
\end{lemme}

\begin{proof}
Since $\phi,\phi' \in \Aut{\compl Q}$ and $f \in \Aut{\compl L}$, we only have to check that $\psi(Q)=Q$.\\
Let $\chi : X \to \p$ (resp. $\chi' : X' \to \p$) be a minimal resolution of the indeterminacies of $\phi$ (resp. $\phi'$) and write $\epsilon:=\phi \circ \chi$ (resp. $\epsilon':=\phi' \circ \chi'$). Call $\E_1,...,\E_m$ (resp. $\E_1',...,\E_m'$) the strict transforms of the exceptional curves of $\chi$ (resp. $\chi'$) in $X$ (resp. $X'$). It follows from Lemma~$\ref{Bla09}$ that $\epsilon(\E_m)=Q$ (resp. $\epsilon'(\E_m')=Q$). Then factorising $\psi$ we get $\psi(Q)=Q$.
\end{proof}

Now we study the automorphisms $\alpha \in \Aut{\p}$ such that $\alpha(C)=D$.

\begin{lemme}
If $\alpha \in \Aut{\p}$ sends $C$ onto $D$, then $a:=(\phi')^{-1} \circ \alpha \circ \phi$ is an automorphism of $\p$ and satisfies :
\begin{enumerate}
	\item $a(L)=L$,
	\item $a(Q)=Q$,
	\item $a(p_i)=p_i'$ for $i=1,...,m$.
\end{enumerate}
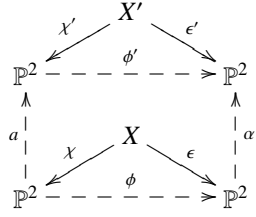
\[\xymatrix@R=3.5mm{
	 & X' \ar[rd]^{\epsilon'} \ar[ld]_{\chi'} &  \\
	\p \ar@{-->}[rr]^{\phi'} & & \p \\
	 & X \ar[rd]^{\epsilon} \ar[ld]_{\chi} &  \\
	\p \ar@{-->}[rr]^{\phi} \ar@{-->}[uu]^{a} & & \p \ar@{-->}[uu]_{\alpha} \\
	}
\]
\end{lemme}

\begin{proof}
Call $q_1,...,q_m$ (resp. $q_1'$, ..., $q_m'$) the points blown-up by $\epsilon$ (resp. $\epsilon'$). Then these points are the singular points of $C$ (resp. $D$). Since $\alpha$ is an automorphism such that $\alpha(C)=D$, then $\alpha$ sends $q_i$ on $q_i'$ for $i=1,...,m$, and lifts to an isomorphism $X \to X'$ which sends $\E_i$ on $\E_i'$ for $i=1,...,m-1$ and $\Q$ on $\Q'$.\\
Since $Q$ is the conic through $q_1,...,q_5$, then $\alpha(Q)=Q$, and the isomorphism $X \to X'$ sends $\E_m$ on $\E_m'$. So $\chi$ and $\chi'$ contract the curves in $X$ and $X'$ which correspond by mean of this isomorphism, and we deduce that $a \in \Aut{\p}$.\\
It follows then that $a$ sends $p_i$ on $p_i'$, $a(Q)=Q$ and that $a(L)=L$.
\end{proof}

\section{The counterexample.}
In this section, we describe more explicitely the construction given in the previous section, by giving more concrete examples.\\
We choose $n\geq 1$ and will define $\Delta\colon X\to \p$ which is the blow-up of some points $p_1,\dots,p_{4+2n}$, such that $p_1\in \p$, and for $i\geq 2$ the point $p_i$ is infinitely near to $p_{i-1}$. We call $E_i$ the exceptional curve associated to $p_i$ and $\tilde{E_i}$ its strict transform in $X$. The points will be choosed so that :
\begin{itemize}
	\item $p_i$ belongs to $Q$ (as proper or infinitely near points) if and only if $i\in \{1,\dots,4\}$,
	\item $p_i$ belongs (as a proper or infinitely near point) to $E_4$ if and only if  $i\in \{5,\dots,4+n\}$,
	\item $p_i\in E_{i-1}\backslash E_{i-2}$ if $i\in \{5+n,\dots, 4+2n\}$.
\end{itemize}

Note that $p_1,\dots,p_{4+n}$ are fixed by these conditions, and that $p_{5+n},\dots, p_{4+2n}$ depends on parameters.
On the surface $X$, we obtain the following dual graph of curves.\\

\begin{center}
\psset{xunit=0.6cm,yunit=0.4cm,algebraic=true,dotstyle=*,dotsize=3pt 0,linewidth=0.8pt,arrowsize=3pt 2,arrowinset=0.25}
\begin{pspicture*}(-2.27,-6.5)(18.88,3)
\psline(-1,2)(1,2)
\psline[linestyle=dashed,dash=3pt 3pt](3,2)(5,2)
\psline(5,2)(7,2)
\psline(7,2)(9,2)
\psline(9,2)(11,2)
\psline(13,2)(15,2)
\psline(7,2)(7,0)
\psline(7,0)(7,-2)
\psline(7,-2)(7,-4)
\psline(7,-4)(7,-6)
\psline(1,2)(3,2)
\psline[linestyle=dashed,dash=3pt 3pt](11,2)(13,2)
\psdots[linecolor=blue](-1,2)
\rput[bl](-0.9,2.17){\blue{$\Q[-1]$}}
\psdots[linecolor=blue](1,2)
\rput[bl](1.09,2.17){\blue{$\E_5$}}
\psdots[linecolor=blue](3,2)
\rput[bl](3.12,2.17){\blue{$\E_6$}}
\psdots[linecolor=blue](5,2)
\rput[bl](5.11,2.17){\blue{$\E_{3+n}$}}
\psdots[linecolor=blue](7,2)
\rput[bl](7.11,2.17){\blue{$\E_{4+n}$}}
\psdots[linecolor=blue](9,2)
\rput[bl](9.11,2.17){\blue{$\E_{5+n}$}}
\psdots[linecolor=blue](11,2)
\rput[bl](11.1,2.17){\blue{$\E_{6+n}$}}
\psdots[linecolor=blue](13,2)
\rput[bl](13.1,2.17){\blue{$\E_{3+2n}$}}
\psdots[linecolor=blue](15,2)
\rput[bl](15.1,2.17){\blue{$\E_{4+2n}[-1]$}}
\psdots[linecolor=blue](7,0)
\rput[bl](7.11,0.17){\blue{$\E_4[-(n+1)]$}}
\psdots[linecolor=blue](7,-2)
\rput[bl](7.11,-1.85){\blue{$\E_3$}}
\psdots[linecolor=blue](7,-4)
\rput[bl](7.11,-3.85){\blue{$\E_2$}}
\psdots[linecolor=blue](7,-6)
\rput[bl](7.11,-5.84){\blue{$\E_1$}}
\end{pspicture*}
\end{center}
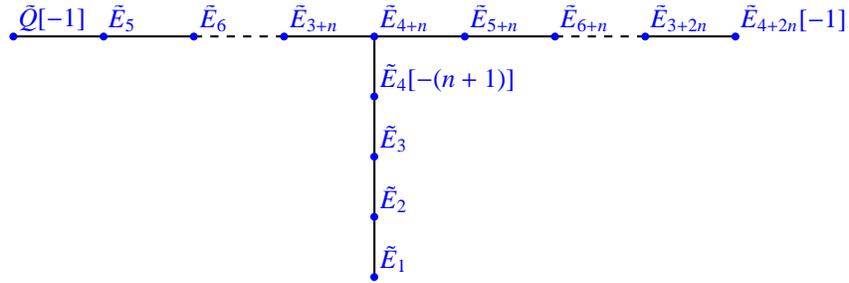
\textsc{Figure 1.} The dual graph of the curves $\E_1,\dots,\E_{3+2n},E_{4+2n}, \Q$. Two curves have an edge between them if and only they intersect, and their self-intersection is written in brackets, if and only if it is not –2.\\

The symmetry of the graph implies the existence of a birational morphism $\epsilon \colon X\to \p$ which contracts the curves $\E_1,\dots,\E_{3+2n}, \Q$, and whichs sends $E_{4+2n}$ on a conic. We may choose that this conic is $Q$, so that $\phi= \epsilon \circ \Delta^{-1}$ restricts to an automorphism of $\compl Q$.\\
Calculating auto-intersection, the image by $\phi$ of a line of the plane which does not pass through $p_1$ has degree $4n+1$.

\subsection{Choosing the points}
Now we are going to choose the birational maps $f$ and the points which define $\phi$ in order to get two curves which give a counterexample to the conjecture of Yoshihara.

We choose that $L$ is the line of equation $z=0$, $Q$ is the conic of equation $xz=y^2$ and $p_1=(0:0:1)$.

We define the birational map $f : \p \dasharrow \p$ by :

\begin{center}
$f(x:y:z)=\left(\mu^2(\lambda xz+ (1-\lambda)y^2):\mu yz:z^2 \right)$
 with $\lambda,\mu\in \mathbb{K}^{*}$ and $\lambda \neq 1$.
\end{center}

The map $f$ preserves $Q$, and is an isomorphism at a local neighbourhood of $p_1$. In consequence, $f$ sends respectively $p_1,\dots,p_{4+2n}$ on some points $p_1',\dots,p_{4+2n}'$ which will define $\Delta '\colon X \to \p$, $\epsilon'\colon X'\to \p$ and $\phi'=\epsilon'\circ (\Delta')^{-1}$ in the same way as $\phi$ was constructed.

We describe now the points $p_i$ and $p_i'$ in local coordinates.\\
Since $f$ preserves $Q$ and fixes $p_1$, we have $p_i'=p_i$ for $i=1,\dots,4$. Locally, the blow-up of $p_1,\dots, p_4$ corresponds to :
$$\phi_4\colon \A^2\to \p \quad , \quad \phi_4(x,y)=(xy^4+y^2:y:1).$$

The curve $E_4$ corresponds to $y=0$, and the conic $\tilde{Q}$ to $x=0$. The lift of $f$ in these coordinates is :
$$(x,y)\mapsto (\lambda\mu^2x,\mu y).$$

The blow-up of the points $p_5,\dots,p_{4+n}$ (which are equal to $p_5',\dots,p_{4+n}'$) now corresponds to :
$$\phi_{4+n}\colon \A^2\to \A^2 \quad , \quad \phi_{4+n}(x,y)=(x,x^ny).$$

So the lift of $f$ corresponds to :
$$(x,y) \mapsto \left( \lambda \mu^2x, \tfrac{y}{\lambda^n \mu^{2n-1}} \right).$$

We set $p_{4+n+i}=(0,a_i)$ for $i\in \{1,\dots, n\}$ with $a_n \neq 0$. The blow-up of $p_{5+n},...,p_{4+n+i}$ now corresponds to :
\begin{center} $\phi_{4+n+i} : \A^2 \to \A^2$ \quad , \quad $\phi_{4+n+i}(x,y)=\left(x,x^iy+P_i(x) \right)$ where $P_i(x)=a_1x^{i-1}+...+a_i$. \end{center}
Since $f$ sends $p_i$ on $p_i'$, we can set $p_{4+n+i}'=(0,b_i)$ for $i\in \{1,\dots, n\}$ with $b_n \neq 0$. The blow-up of $p_{5+n}',...,p_{4+n+i}'$ then corresponds to :
\begin{center} $\phi_{4+n+i}' : \A^2 \to \A^2$ \quad , \quad $\phi_{4+n+i}'(x,y)=\left(x,x^iy+Q_i(x) \right)$ where $Q_i(x)=b_1x^{i-1}+...+b_i$. \end{center}
So the lift of $f$ corresponds to :
$$(x,y) \mapsto \left( \lambda \mu^2x, \tfrac{x^iy+P_i(x)-\lambda^i \mu^{2i-1}Q_i(\lambda\mu^2x)}{\lambda^i \mu^{2i-1}x^i} \right).$$

The curves $E_{4+n+i}$ and $E_{4+n+i}'$ correspond to $x=0$ in both local charts. Since $f$ is a local isomorphism which sends $p_i$ on $p_i'$ for each $i$, it has to be defined on the line $x=0$. Because $P_i$ and $Q_i$ have both degree $i-1$, this implies that:
\begin{center}
$P_i(x)=\lambda^i \mu^{2i-1}Q_i(\lambda\mu^2x)$ for $i=1,...,n$.
\end{center}
In particular, the coefficients satisfy :
\begin{center}
$a_i=\lambda^i \mu^{2i-1}b_i$ for $i=1,...,n$.
\end{center}
\subsection{The counterexample.}

Now to get a counter example, we must show that any automorphism $a : \p \to \p$ such that $a(L)=L$, $a(Q)=Q$ and $a(p_1)=p_1$ doesn't send $p_i$ on $p_i'$ for at least one $i\in \{5+n,\dots, 4+2n\}$. Let's start with the following Lemma :
\begin{lemme}
Let $a : \p \to \p$ be an automorphism such that $a(L)=L$, $a(Q)=Q$ and $a(p_1)=p_1$. Then $a$ is of the form :
$$a(x:y:z)=\left( k^2x : k y : z \right) \quad \mbox{where } k \in \K^{\ast}.$$
\end{lemme}

\begin{proof}
Follows from a direct calculation.
\end{proof}

\begin{thm}
If $n \geq 2$, the curves $C$ and $D$ obtained with the construction of the previous section give a counter example to the conjecture.
\end{thm}

\begin{proof}
Choose $a_n=a_{n-1}=1$.\\
Since $a$ is an automorphism, it lifts to an automorphism which sends $E_{4+n+i}$ on $E_{4+n+i}'$. Put $\lambda=1$ and $\mu=k$ in the formula for $f$. Then this lift corresponds to :
$$(x,y) \mapsto \left( k^2x, \tfrac{x^iy+P_i(x)-k^{2i-1}Q_i(k^2x)}{k^{2i-1}x^i} \right)$$
where $P_i$ and $Q_i$ are the polynomials defined above.\\
Since $E_{4+n+i}$ and $E_{4+n+i}'$ both correspond to $x=0$ in local charts, this lift has to be well defined on $x=0$. So since $P_i$ and $Q_i$ both have degree $i-1$, we get :
\begin{center}
$P_i(x)=k^{2i-1}Q_i(k^2x)$ for $i=1,...,n$
\end{center}
and the constant terms satisfy $a_i=k^{2i-1}b_i$ for $i=1,...,n$.\\
Since $a_n,a_{n-1} \neq 0$, then $b_n,b_{n-1} \neq 0$. As explained in the previous section, $a$ sends $p_i$ on $p_i'$, so we get :
\begin{center}
$\lambda^i \mu^{2i-1}b_i=k^{2i-1}b_i$ for $i=1,...,n$.
\end{center}
This formula for $i=n$ and $i=n-1$ gives $\lambda=1$ or $\mu=0$, which leads to a contradiction.
\end{proof}
\section{Conclusion.}

We conclude observing that the curves $C$ and $D$ of the previous construction have degree $4n+1$ (using Figure $1$) and are of type I. In particular, we get a counterexample with a curve of degree 9 when $n=2$. One can check by direct computation that the conjecture holds for irreducible curves of type I up to degree 5, because there's only one curve of degree 5 which is of type I and not of type II, up to automorphism of $\p$. One can also check that all irreducible curves of type I of degree 6, 7 and 8 are of type II. So the curves of degree 9 given by this construction leads to a counterexample of minimal degree among the curves of type I.

If we consider the conjecture for all rational curves, the counterexamples in \cite{Bla09} are of degree 39 (and not of type I). So we have new counterexamples with curves of lower degree. It seems that the curves of degree 9 give counterexamples of minimal degree among the rational curves, but it hasn't been shown yet.


\end{document}